\documentclass[12pt, a4paper]{amsart}
\usepackage{fullpage}   % For smaller margins to save paper when printing
\usepackage{hyperref}	% For hyperreferencs
\hypersetup{colorlinks=true, urlcolor=blue}
\newcommand{\Leinartas}{Le{\u\i}nartas\ }
\newcommand{\Leinartass}{Le{\u\i}nartas's\ }
\newcommand{\hl}[1]{\textbf{#1}}
\newcommand{\Kbar}{\overline{K}}
\newcommand{\CC}{\mathbb{C}}
\newcommand{\QQ}{\mathbb{Q}}
\newcommand{\QQbar}{\overline{\mathbb{Q}}}
\newcommand{\NN}{\mathbb{N}}
\newtheoremstyle{plain} % name
     {2ex}              % space above
     {2ex}              % space below
     {}                 % body font; \itshape is common
     {}                 % indent amount (empty = no indent,
     {\bfseries}        % theorem head font
     {}                 % punctuation after thm head
     {1ex}              % space after theorem head: " " =
     {\thmname{#1 }\thmnumber{#2}\thmnote{ \normalfont{(#3)}}.
}
\theoremstyle{plain}
\newtheorem{definition}{Definition}[section]
\newtheorem{remark}[definition]{Remark}

\newtheorem{lemma}[definition]{Lemma}
\newtheorem{theorem}[definition]{Theorem}
\newtheorem{corollary}[definition]{Corollary}

\newtheorem{example}[definition]{Example}

\numberwithin{equation}{section}
\begin{document}
\title{Le{\u\i}nartas's Partial Fraction Decomposition}
\author{Alexander Raichev}
\address{
    Department of Computer Science \\
    University of Auckland \\ 
    Auckland 1001, New Zealand
}
\email{\href{mailto:raichev@cs.auckland.ac.nz}{raichev@cs.auckland.ac.nz}}
\date{25 June 2012}
\begin{abstract}
These notes describe \Leinartass algorithm for multivariate partial fraction decompositions and employ an implementation thereof in Sage.
\end{abstract}
\maketitle

%==============================================================================
%%%% Main matter
\section{Introduction}

In \cite{Lein1978}, \Leinartas gave an algorithm for decomposing multivariate rational expressions into partial fractions.
In these notes I re-present \Leinartass algorithm, because it is not well-known, because its English translation \cite{Lein1978} is difficult to find, and because it is useful e.g. for computing residues of multivariate rational functions; see \cite[Chapter 3]{AiYu1983} and \cite{RaWi2012}.

Along the way I include examples that employ an open-source implementation of \Leinartass algorithm that I wrote in Sage \cite{Sage}.
The code can be downloaded from \href{http://www.alexraichev.org/research.html}{my website} and is currently under peer review for incorporation into the Sage codebase.

For a different type of multivariate partial fraction decomposition, one that uses iterated univariate partial fraction decompositions, see \cite{Stou2008}. 
%------------------------------------------------------------------------------
\section{Algorithm}

Henceforth let $K$ be a field and $\Kbar$ its algebraic closure.
We will work in the factorial polynomial rings $K[X]$ and $\Kbar[X]$, where $X = X_1,\ldots, X_d$ with $d \ge 1$.
\Leinartass algorithm is contained in the constructive proof of the following theorem, which is \cite[Theorem 1]{Lein1978}\footnote{
\Leinartas used $K = \CC$, but that is an unnecessary restriction.
By the way, \Leinartass article contains typos in equation (c) on the second page, equation (b) on the third page, and the equation immediately after equation (d) on the third page: the right sides of those equations should be multiplied by $P$.
}.

\begin{theorem}[\Leinartas decompositon]\label{leinartas-decomp}
Let $f = p/q$, where $p, q \in K[X]$. 
Let $q = q_1^{e_1} \cdots q_m^{e_m}$ be the unique factorization of $q$ in $K[X]$, and let $V_i = \{x \in \Kbar^d : q_i(x) = 0 \}$, the algebraic variety of $q_i$ over $\Kbar$.

The rational expression $f$ can be written in the form
\[
    f = \sum_A \frac{p_A}{\prod_{i \in A} q_i^{b_i}},   
\]
where the $b_i$ are positive integers (possibly greater than the $e_i$), the $p_A$ are polynomials in $K[X]$ (possibly zero), and the sum is taken over all subsets $A \subseteq \{1,\ldots, m\}$ such that $\cap_{i \in A} V_i \neq \emptyset$ and $\{q_i : i \in A\}$ is algebraically independent (and necessarily $|A| \le d$). 
\end{theorem}

Let us call a decomposition of the form above a \hl{\Leinartas decomposition}.
An immediate consequence of the theorem is the following.

\begin{corollary}
Every rational expression in $d$ variables can be represented as a sum of rational expressions each of whose denominators contains at most $d$ unique irreducible factors.    
\qed
\end{corollary} 

Now for a constructive proof of the theorem.
It involves two steps: decomposing $f$ via the Nullstellensatz and then decomposing each resulting summand via algebraic dependence.
We need a few lemmas.

The following lemma is a strengthening of the weak Nullstellensatz and is proved in \cite[Lemma 3.2]{DLMM2008}. 

\begin{lemma}[Nullstellensatz certificate]\label{null-cert}
A finite set of polynomials $\{q_1, \ldots, q_m \} \subset K[X]$ has no common zero in $\Kbar^d$ iff there exist polynomials $h_1, \ldots, h_m \in K[X]$ such that
\[
    1 = \sum_{i=1}^m h_i q_i. %\qquad (\star)
\]   
Moreover, if $K$ is a computable field, then there is a computable procedure to check whether or not the $q_i$ have a common zero in $\Kbar^d$ and, if not, return the $h_i$. 
\qed
\end{lemma}

Let us call a sequence of polynomials $h_i$ satisfying the equation above a \hl{Nullstellensatz certificate} for the $q_i$.
Note that in contrast to the usual weak Nullstellensatz, here the polynomials $h_i$ are in $K[X]$ and not just in $\Kbar[X]$.
%This lemma is proved in \cite{DLMM2008} by taking a known computable bound on the total degrees of the $h_i$ and solving a linear system of equations to find the coefficients of the $h_i$ in $(\star)$ which will necessarily lie in $K$ since the coefficients of the $q_i$ lie in $K$.

Some examples of computable fields are finite fields, $\QQ$, finite degree extensions of $\QQ$, and $\QQbar$.

%\begin{example}
% \begin{verbatim}
%     
%     
%     sage: R.<x,y> = PolynomialRing(QQ)   # Define our polynomial ring.
%     sage: f = 1/x + 1/y + 1/(x*y + 1); f
%     (x^2*y + x*y^2 + x*y + x + y)/(x^2*y^2 + x*y)
%     sage: ff = to_refd(f); ff   # Format f 
%     (x^2*y + x*y^2 + x*y + x + y, [(y, 1), (x, 1), (x*y + 1, 1)])
%     sage: qs = ff.denominator_list(); qs
%     [(y, 1), (x, 1), (x*y + 1, 1)]
%     sage: L = ff.nullstellensatz_certificate(); L
%     [0, -y, 1]
%     sage: # Polynomial combination from L of qs should equal 1.
%     sage: sum([L[i]*qs[i][0]**qs[i][1] for i in range(len(qs))]) == 1
%     True
%     
% \end{verbatim}
%\end{example}

Applying Lemma~\ref{null-cert} we get the following lemma \cite[Lemma 3]{Lein1978}.

\begin{lemma}[Nullstellensatz decomposition]\label{null-decomp}
Under the hypotheses of Theorem~\ref{leinartas-decomp},
the rational expression $f$ can be written in the form
    \[
        f = \sum_A \frac{p_A}{\prod_{i \in A} q_i^{e_i}},
    \]
where the $p_A$ are polynomials in $K[X]$ (possibly zero) and the sum is taken over all subsets $A \subseteq \{1,\ldots, m\}$ such that $\cap_{i \in A} V_i \neq \emptyset$.
\end{lemma}

\begin{proof}
If $\cap_{i=1}^m V_i \neq \emptyset$, then the result holds.

Suppose now that $\cap_{i=1}^m V_i = \emptyset$.
Then the polynomials $q_i^{e_i}$ have no common zero in $\Kbar^d$. 
So by Lemma~\ref{null-cert}
\[
    1 = h_1 q_1^{e_1} + \cdots + h_m q_m^{e_m}
\]
for some polynomials $h_i$ in $K[X]$.
Multiplying both sides of the equation by $p/q$ yields
\begin{align*}
    f 
    &= 
    \frac{p (h_1 q_1^{e_1} + \cdots + h_m q_m^{e_m})}{q_1^{e_1} \cdots  
    q_m^{e_m}} \\
    &=
    \sum_{i=1}^m \frac{p h_i}{q_1^{e_1} \cdots \widehat{q_i^{e_i}} \cdots 
    q_m^{e_m}}
\end{align*}
Note that $p h_i \in K[X]$.

Next we check each summand $p h_i/(q_1^{e_1} \cdots \widehat{q_i^{e_i}} \cdots q_m^{e_m})$ to see whether $\cap_{j \neq i } V_j \neq \emptyset$.
If so, then stop.
If not, then apply Lemma~\ref{null-cert} to ${q_1^{e_1}, \ldots \widehat{q_i^{e_i}}, \ldots q_m^{e_m}}$.

Repeating this procedure until it stops yields the desired result.
The procedure must stop, because each $V_i \neq \emptyset$ since each $q_i$ is irreducible in $K[X]$ and hence not a unit in $K[X]$.
\end{proof}

Let us call a decomposition of the form above a \hl{Nullstellensatz decomposition}.

\begin{example}
% Continuing the previous example...
% \begin{verbatim}
%     
%     sage: f = 1/x + 1/y + 1/(x*y + 1)
%     sage: dn = pfd(f, kind='nullstellensatz'); dn
%     [1/(x*y + 1), (x + y)/(x*y)]
%     sage: sum(dn) == f   # Decomposition sums to f?
%     True
%         
% \end{verbatim}
Consider the rational expression
\[
    f := \frac{X^2 Y + X Y^2 + X Y + X + Y}{X Y (X Y + 1)} 
\]
in $\QQ(X,Y)$.
Let $p$ denote the numerator of $f$.
The irreducible polynomials $X, Y, XY + 1 \in \QQ[X, Y]$ in the denominator have no common zero in $\QQbar^2$.
So they have a Nullstellensatz certificate, e.g. $(-Y, 0, 1)$: 
\[
    1 = (-Y)X + (0)X + (1)(XY + 1). 
\]

Applying the algorithm in the proof of Lemma~\ref{null-decomp} gives us a Nullstellensatz decomposition for $f$ in one iteration:
\begin{align*}
    f 
    =& \frac{p(-Y)}{Y(XY + 1)} + \frac{p(1)}{XY} \\
    =& \frac{-p}{XY + 1} + \frac{p}{XY} \\
    =& -X -Y -1 + \frac{1}{X Y + 1} + X + Y + 1 + \frac{X + Y}{XY} \\
     & \text{(after applying the division algorithm)} \\
    =& \frac{1}{X Y + 1} + \frac{X + Y}{XY}.
\end{align*}
Notice that 
\[
    f = \frac{1}{X} + \frac{1}{Y} + \frac{1}{XY + 1}
\]
is also a Nullstellensatz decomposition for $f$.
So Nullstellensatz decompositions are not unique.
\end{example}

The next lemma is a classic in computational commutative algebra; see e.g. \cite{Kaya2009}.

\begin{lemma}[Algebraic dependence certificate]\label{algdep-cert}
Any set $S$ of polynomials in $K[X]$ of size $> d$ is algebraically dependent.
Moreover, if $K$ is a computable field and $S$ is finite, then there is a computable procedure that checks whether or not $S$ is algebraically dependent and, if so, returns an annihilating polynomial over $K$ for $S$.     
\qed
\end{lemma}

The next lemma is \cite[Lemma 1]{Lein1978}.
 
\begin{lemma}\label{algdep-powers}
A finite set of polynomials $\{q_1, \ldots, q_m\} \subset K[X]$ is algebraically dependent iff for all positive integers $e_1, \ldots, e_m$ the set of polynomials $\{q_1^{e_1}, \ldots, q_m^{e_m}\}$ is algebraically dependent.    
\end{lemma}

\begin{proof}
A set of polynomials $\{q_1, \ldots, q_m\} \subset K[X]$ is algebraically independent 
iff the $m \times d$ Jacobian matrix $J(q_1, \ldots, q_m) := \left( \frac{\partial q_i}{\partial X_j}\right)$ over the vector space $K(X)^d$ has rank $m$ (by the Jacobian criterion; see e.g. \cite{EhRo1993}) 
iff for all positive integers $e_i$ the matrix $\left(e_i q_i^{e_i -1} \frac{\partial q_i}{\partial X_j}\right) = J(q_1^{e_1}, \ldots, q_m^{e_m})$ over the vector space $K(X)^d$ has rank $m$ (since we are just taking scalar multiples of rows) iff the set of polynomials $q_1^{e_1}, \ldots, q_m^{e_m}$ is algebraically independent (by the Jacobian criterion).

Moreover, if $\{q_1, \ldots, q_m\}$ is algebraically dependent, then any member of the (necessarily nonempty) elimination ideal 
\[
\langle Y_1 - q_1, \ldots, Y_m - q_m, Y_1^{e_1} - Z_1, \ldots, Y_m^{e_m} - Z_m \rangle_{K[X,Y,Z]} \cap K[Z_1, \ldots, Z_m],
\]
is an annihilating polynomial for $q_1^{e_1}, \ldots, q_m^{e_m}$. 
Moreover a finite basis for the elimination ideal can be computed using Groebner bases; see e.g. \cite[Chapter 3]{CLO2007}.
\end{proof}

%\begin{example}
% \begin{verbatim}
%     
%     
%     sage: R.<x,y,z> = PolynomialRing(QQ)   
%     sage: f = 1/x + 1/y + 1/z + 1/(x*y + z); f
%     (x^2*y^2 + x^2*y*z + x*y^2*z + 2*x*y*z + x*z^2 + y*z^2)/
%     (x^2*y^2*z + x*y*z^2)
%     sage: ff = to_refd(f); ff
%     (x^2*y^2 + x^2*y*z + x*y^2*z + 2*x*y*z + x*z^2 + y*z^2, 
%     [(z, 1), (y, 1), (x, 1), (x*y + z, 1)])
%     sage: qs = ff.denominator_list()
%     sage: J = ff.algebraic_dependence_certificate(); J
%     Ideal (T0 - T3 + T1*T2) of Multivariate Polynomial Ring in T0, T1, T2, T3
%     over Rational Field
%     sage: g = J.gens()[0];   # Generator of J annihilates qs?
%     sage: g(*(q**e for q, e in qs)) == 0
%     True
%         
% \end{verbatim}
%\end{example}

Applying the previous two lemmas we get our final lemma \cite[Lemma 2]{Lein1978}.

\begin{lemma}[Algebraic dependence decomposition]\label{algdep-decomp}
Under the hypotheses of Theorem~\ref{leinartas-decomp},
the rational expression $f$ can be written in the form
    \[
        f = \sum_A \frac{p_A}{\prod_{i \in A} q_i^{b_i}}, 
    \]
where the $b_i$ are positive integers (possibly greater than the $e_i$), the $p_A$ are polynomials in $K[X]$ (possibly zero), and the sum is taken over all subsets $A \subseteq \{1,\ldots, m\}$ such that $\{q_i : i \in A\}$ is algebraically independent (and necessarily $|A| \le d$).
\end{lemma}

\begin{proof}
If $\{q_1, \ldots, q_m\}$ is algebraically independent, then the result holds.
Notice that in this case $m \le d$ by Lemma~\ref{algdep-cert}.

Suppose now that $\{q_1, \ldots, q_m\}$ is algebraically dependent.
Then so is $\{q_1^{e_1}, \ldots, q_m^{e_m}\}$ by Lemma~\ref{algdep-powers}.
Let $g = \sum_{\nu \in S} c_\nu Y^\nu \in K[Y_1, \ldots, Y_m]$ be an annihilating polynomial for $\{q_1^{e_1}, \ldots, q_m^{e_m}\}$, where $S \subset \NN^m$ is the set of multi-indices such that $c_\nu \neq 0$.
Choose a multi-index $\alpha \in S$ of smallest norm $||\alpha|| = \alpha_1 + \cdots + \alpha_m$.
Then at $Q:= (q_1^{e_1}, \ldots, q_m^{e_m})$ we have
\begin{align*}
    g(Q) 
    &= 0 \\
    c_\alpha Q^\alpha 
    &= -\sum_{\nu \in S \setminus{\{\alpha\}}} c_\nu Q^\nu \\
    1 
    &= \frac{-\sum_{\nu \in S \setminus{\{\alpha\}}} c_\nu Q^\nu}{c_\alpha Q^\alpha}.
\end{align*}
Multiplying both sides of the last equation by $p / q$ yields
\begin{align*}
    \frac{p}{q} 
    &= \sum_{\nu \in S \setminus{\{\alpha\}}} 
    \frac{-p c_\nu Q^\nu}{c_\alpha Q^{\alpha + 1}} \\
    &= 
    \sum_{\nu \in S \setminus{\{\alpha\}}} \frac{-p c_\nu}{c_\alpha}
    \prod_{i=1}^m \frac{q_i^{e_i \nu_i}}{q_i^{e_i(\alpha_i + 1)}} \\
\end{align*}
Since $\alpha$ has the smallest norm in $S$ it follows that for any $\nu \in S \setminus{\{\alpha\}}$ there exists $i$ such that $\alpha_i + 1 \le \nu_i$, so that $e_i(\alpha_i + 1) \le e_i \nu_i$.
So for each $\nu \in S \setminus{\{\alpha\}}$, some polynomial $q_i^{e_i (\alpha_i + 1)}$ in the denominator of the right side of the last equation cancels.

Repeating this procedure yields the desired result.
\end{proof}

Let us call a decomposition of the form above an \hl{algebraic dependence decomposition}.

\begin{example}
% Continuing the previous example...
% \begin{verbatim}
%     
%     sage: f = 1/x + 1/y + 1/z + 1/(x*y + z)
%     sage: da = pfd(f, kind='algebraic_dependence', factor=True); da
%     [(-x^2*y^2 - x^2*y*z - x*y^2*z - 2*x*y*z - x*z^2 - y*z^2, 
%     [(z, 2), (x*y + z, 1)]), 
%     (x^2*y^2 + x^2*y*z + x*y^2*z + 2*x*y*z + x*z^2 + y*z^2, 
%     [(z, 2), (y, 1), (x, 1)])]
%         
% \end{verbatim}
Consider the rational expression
\[
    f := \frac{(X^2 Y^2 + X^2 Y Z + X Y^2 Z + 2 X Y Z + X Z^2 + Y Z^2)}{X Y Z (X Y + Z)} 
\]
in $\QQ(X,Y,Z)$.
Let $p$ denote the numerator of $f$.
The irreducible polynomials $X, Y, Z, XY + Z \in \QQ[X,Y,Z]$ in the denominator are four in number, which is greater than the number of ring indeterminates, and so they are algebraically dependent.
An annihilating polynomial for them is $g(A,B,C,D) = AB  + C - D$.

Applying the algorithm in the proof of Lemma~\ref{algdep-decomp} gives us an algebraic dependence decomposition for $f$ in one iteration:
\begin{align*}
    f
    =& \sum_{\nu \in S \setminus{\{\alpha\}}} 
       \frac{-p c_\nu Q^\nu}{c_\alpha Q^{\alpha + 1}} \\
     & \text{where $Q = (X,Y,Z,XY + Z)$ and $\alpha = (0,0,0,1)$} \\
    =& \frac{pQ^{(1,1,0,0)}}{Q^{(1,1,1,2)}} + \frac{pQ^{(0,0,1,0)}}{Q^{(1,1,1,2)}} \\
    =& \frac{p}{Q^{(0,0,1,2)}} + \frac{p}{Q^{(1,1,0,2)}} \\
    =& \frac{p}{Z (XY + Z)^2} + \frac{p}{XY(XY + Z)^2}.
\end{align*}

Notice that in this example the exponent 2 of the irreducible factor $XY + Z$ in the denominators of the decomposition is larger than the exponent 1 of $XY + Z$ in the denominator of $f$.
Notice also that 
\[
    f = \frac{1}{X} + \frac{1}{Y} + \frac{1}{Z} + \frac{1}{XY + Z}
\]
is also an algebraic dependence decomposition for $f$.
So algebraic dependence decompositions are not unique.
\end{example}

Finally, here is \Leinartass algorithm.

\begin{proof}[Proof of Theorem~\ref{leinartas-decomp}]
First find the irreducible factorization of $q$ in $K[X]$.
This is a computable procedure if $K$ is computable.
Then decompose $f$ via Lemma~\ref{null-decomp}.
Finally decompose each summand of the result via Lemma~\ref{algdep-decomp}.     
As highlighted above, the last two steps are computable if $K$ is.
\end{proof}

\begin{example}
Consider the rational expression
\[
    f := \frac{2X^2 Y + 4X Y^2 + Y^3 - X^2 - 3 X Y - Y^2}{X Y (X + Y) (Y - 1)}
\]
in $\QQ(X,Y)$.
Computing a Nullstellensatz decomposition according to the proof of Lemma~\ref{null-decomp} with Nullstellensatz combination $1 = 0(X) + 1(Y) + 0(X + Y) -1(Y - 1)$ yields
\begin{align*}
    f =& X - Y + \frac{Y^3 + X^2 - Y^2 + X}{X(Y-1)} + 
         \frac{X^2 Y - 2X^2 - XY}{(X + Y)(Y - 1)} +\\
       & \frac{-2X^3 - Y^3 - 2X^2 + Y^2}{X(X + Y)} + 
         \frac{2X^2 Y - Y^3 + X^2 + 3X Y + Y^2}{XY(X + Y)}.
\end{align*}

Computing an algebraic dependence decomposition for the last term according to the proof of Lemma~\ref{algdep-decomp} with annihilating polynomial $g(A,B,C) = A + B - C$ for $(X, Y, X + Y)$ yields
\begin{align*}
    & \frac{2X^2 Y - Y^3 + X^2 + 3X Y + Y^2}{XY(X + Y)} \\
    &= 1 + \frac{2X^2 Y - Y^3 + X^2 + 3X Y + Y^2}{XY^2} + 
      \frac{-2X^2 Y - XY^2 - X^2 - 3XY - Y^2}{Y^2 (X + Y)}.
\end{align*}

The two equalities taken together give us a \Leinartas decomposition for $f$.

Notice that
\[
    f =  \frac{1}{X} + \frac{1}{Y} + \frac{1}{X + Y} + \frac{1}{Y - 1}
\]
is also a \Leinartas decomposition of $f$.
So \Leinartas decompositions are not unique.
\end{example}

\begin{remark}
In case $d=1$, \Leinartas decompositions are unique once the fractions are written in lowest terms (and one disregards summand order).
To see this, note that a \Leinartas decomposition of a univariate rational expression $f = p/q$ must have fractions all of the form $p_i/q_i^{e_i}$, where $q = q_1^{e_1} \cdots q_m^{e_m}$ is the unique factorization of $q$ in $K[X]$.
This is because two or more univariate polynomials are algebraically dependent (by Lemma~\ref{algdep-cert}).
Assume without loss of generality here that $\deg(p) < \deg(q)$.
It follows that if we have two \Leinartass decompositions of $p/q$, then we can write them in the form $a_1/q' + a_2/q'' = b_1/q' + b_2/q''$, where $q = q'q''$ with $q'$ and $q''$ coprime, $\deg(a_1), \deg(b_1) < \deg(q')$, and $\deg(a_2), \deg(b_2) < \deg(q'')$.
Multiplying the equality by $q$ we get $a_1q'' + a_2q' = b_1q'' + b_2q'$. 
So $a_1 \equiv b_1 \pmod{q'}$ and $a_2 \equiv b_2 \pmod{q''}$.
Thus $a_1 = b_1$ and $a_2 = b_2$.
This observation used inductively demonstrates uniqueness.

This argument fails in case $d \ge 2$, because then a \Leinartas decomposition might not have fractions all of the form $p_i/q_i^{e_i}$.
\end{remark}	

\begin{remark}
A rational expression already with $\cap_{i=1}^m V_i \neq \emptyset$ and $\{q_1, \ldots, q_m\}$ algebraically independent, can not necessarily be decomposed further into partial fractions.
For example,
\[
	f = \frac{1}{X_1 X_2 \cdots X_m} \in K(X_1, X_2, \ldots, X_d),
\] 	
with $m \le d$ can not equal a sum of rational expressions whose denominators each contain fewer than $m$ of the $X_i$.
Otherwise, multiplying the equation by $X_1 X_2 \cdots X_m$ would yield 
\[
	1 = \sum_{i\in B} h_i X_i
\]
for some $h_i \in K[X]$ and some nonempty subset $B\subseteq \{1, 2, \ldots, m\}$, a contradiction to Lemma~\ref{null-cert} since $\{X_i : i\in B\}$ have a common zero in $\Kbar^d$, namely the zero tuple.
\end{remark}
% %------------------------------------------------------------------------------
% \section{Examples}
% To finish I present some more examples of \Leinartass algorithm in action. 
% 
% \begin{example}
% \begin{verbatim}
% 
% 
% sage: R.<x> = PolynomialRing(QQ)
% sage: f = 1 + 2/x + 2*x/(x^2 - 4*x + 8); f
% (x^3 + 16)/(x^3 - 4*x^2 + 8*x)
% sage: d1 = pfd(f, decomp='leinartas', factor=True); d1
% [(1, []), (2*x, [(x^2 - 4*x + 8, 1)]), (2, [(x, 1)])]
% sage: d2 = pfd(f, decomp='leinartas'); d2
% [1, 2*x/(x^2 - 4*x + 8), 2/x]
% sage: sum(d2) == f
% True
% 
% \end{verbatim}
% \end{example}
% 
% 
% 
% \begin{example}
% \begin{verbatim}
% 
% 
% sage: R.<x,y,z>= PolynomialRing(GF(2, 'a'))
% sage: R.base_ring()
% Finite Field of size 2
% sage: f = x + 1/x + 1/y + 1/z + 1/(x*y + z); f
% (x^3*y^2*z + x^2*y*z^2 + x^2*y^2 + x^2*y*z + x*y^2*z + x*z^2 +
% y*z^2)/(x^2*y^2*z + x*y*z^2)
% sage: d1 = pfd(f, decomp='leinartas', factor=True); d1
% [(x, []), (x^2*y^2 + x^2*y*z + x*y^2*z + x*z^2 + y*z^2, [(z, 2), (x*y +
% z, 1)]), (x^2*y^2 + x^2*y*z + x*y^2*z + x*z^2 + y*z^2, [(z, 2), (y, 1),
% (x, 1)])]
% sage: d2 = pfd(f, decomp='leinartas'); d2
% [x, (x^2*y^2 + x^2*y*z + x*y^2*z + x*z^2 + y*z^2)/(x*y*z^2 + z^3),
% (x^2*y^2 + x^2*y*z + x*y^2*z + x*z^2 + y*z^2)/(x*y*z^2)]
% sage: sum(d2) == f
% True
% 
% \end{verbatim}
% \end{example}
%==============================================================================
%%% Back matter
\bibliographystyle{amsalpha}
\bibliography{combinatorics}
\end{document}